\documentclass[a4paper,12pt]{article}
\usepackage{amssymb}
\usepackage{amsfonts}
\usepackage{amsmath,graphicx}
\usepackage{amsmath,amssymb,graphicx}
\usepackage{multicol}
\usepackage{color}
\usepackage[a4paper,lmargin=3cm,rmargin=3cm,tmargin=3cm,
bmargin=3cm]{geometry}

\newtheorem{theorem}{Theorem}

\newtheorem{corollary}[theorem]{Corollary}

\newtheorem{lemma}[theorem]{Lemma}

\newtheorem{problem}[theorem]{Problem}

\newenvironment{proof}[1][Proof]{\noindent \textbf{#1.} }{\  \rule{0.5em}{0.5em}}

\begin{document}

\title{New subclass of the class of close-to-convex harmonic mappings
defined by a third-order differential inequality}
\author{Serkan \c{C}akmak, Elif Ya\c{s}ar and Sibel Yal\c{c}\i n \\
Department of Mathematics, Faculty of Arts and Sciences, \\
Bursa Uludag University, 16059, G\"{o}r\"{u}kle, Bursa, Turkey.\\
E-mail: serkan.cakmak64@gmail.com, elifyasar@yahoo.com \\
and syalcin@uludag.edu.tr}
\maketitle

\begin{abstract}
In this paper, we introduce a new subclass of harmonic functions $\mathfrak{f%
}=\mathfrak{s}+\overline{\mathfrak{t}}$ in the open unit disk $\mathcal{U}%
=\left \{ z\in \mathbb{C}:\left \vert z\right \vert <1\right \} $ satisfying%
\newline

\noindent ${\text{Re}}\left[ \gamma \mathfrak{s}^{\prime }(z)+\delta z%
\mathfrak{s}^{\prime \prime }(z)+\left( \frac{\delta -\gamma }{2}\right)
z^{2}\mathfrak{s}^{\prime \prime \prime }\left( z\right) -\lambda \right]
>\left \vert \gamma \mathfrak{t}^{\prime }(z)+\delta z\mathfrak{t}^{\prime
\prime }(z)+\left( \frac{\delta -\gamma }{2}\right) z^{2}\mathfrak{t}%
^{\prime \prime \prime }\left( z\right) \right \vert,$\newline
\noindent where $0\leq \lambda <\gamma \leq \delta, z\in \mathcal{U}.$ We
determine several properties of this class such as close-to-convexity,
coefficient bounds, and growth estimates. We also prove that this class is
closed under convex combination and convolution of its members. Furthermore,
we investigate the properties of fully starlikeness and fully convexity of
the class.

\end{abstract}

\section{Introduction}

\bigskip Let $\mathcal{H}$ denote the class of complex-valued harmonic
functions $\mathfrak{f}=\mathfrak{s}+\overline{\mathfrak{t}}$ defined in the
open unit disk $\mathcal{U}=\left \{ z\in \mathbb{C}:\left \vert
z\right
\vert <1\right \} ,$ and normalized by $\mathfrak{f}(0)=\mathfrak{f}%
_{z}(0)-1=0$. Also, let $\mathcal{H}^{0}=\left \{ \mathfrak{f}\in \mathcal{H}%
:\mathfrak{f}_{\overline{z}}(0)=0\right \} $. Each function $\mathfrak{f}\in 
\mathcal{H}^{0}$ can be expressed as $\mathfrak{f}=\mathfrak{s}+\overline{%
\mathfrak{t}},$ where%
\begin{equation}
\mathfrak{s}(z)=z+\overset{\infty }{\underset{m=2}{\sum }}a_{m}z^{m},\text{
\ \ \ }\mathfrak{t}(z)=\overset{\infty }{\underset{m=2}{\sum }}b_{m}z^{m}
\label{eqH}
\end{equation}%
are analytic in $\mathcal{U}$. A necessary and sufficient condition for $%
\mathfrak{f}$ to be locally univalent and sense-preserving in $\mathcal{U}$
is that $\left \vert \mathfrak{s}^{\prime }(z)\right \vert >\left \vert 
\mathfrak{t}^{\prime }(z)\right \vert $ in $\mathcal{U}$. See \cite%
{Clunie,Durenharm}.

Denote by $\mathcal{S}_{H}$ the class of functions $\mathfrak{f}=\mathfrak{s}%
+\overline{\mathfrak{t}}$ that are harmonic, univalent and sense-preserving
in the unit disk $\mathcal{U}$. Further, let $\mathcal{S}_{H}^{0}=\left \{ 
\mathfrak{f}\in \mathcal{S}_{H}:\mathfrak{f}_{\overline{z}}(0)=0\right \} .$
Note that, with $\mathfrak{t}(z)=0,$ the classical family $\mathcal{S}$ of
analytic univalent and normalized functions in $\mathcal{U}$ is a subclass
of $\mathcal{S}_{H}^{0}$ $,$ just as the family $\mathcal{A}$ of analytic
and normalized functions in $\mathcal{U}$ is a subclass of $\mathcal{H}^{0}.$
A simply connected subdomain of $\mathbb{C}$ is said to be close-to-convex
if its complement in $\mathbb{C}$ can be written as the union of
non-crossing half-lines.

Let $\mathcal{K},\mathcal{S}^{\ast }$ and $\mathcal{C}$ be the subclasses of 
$\mathcal{S}$ mapping $\mathcal{U}$ onto convex, starlike and
close-to-convex domains, respectively, just as $\mathcal{K}_{H}^{0},$ $%
\mathcal{S}_{H}^{\ast ,0}$ and $\mathcal{C}_{H}^{0}$ are the subclasses of $%
\mathcal{S}_{H}^{0}$ mapping $\mathcal{U}$ onto their respective domains.

In \cite{Hernandez}, Hernandez and Martin introduced the notion of stable
harmonic mappings. A sense-preserving harmonic mapping $\mathfrak{f}=%
\mathfrak{s}+\overline{\mathfrak{t}}$ is said to be stable harmonic
univalent (resp. stable harmonic convex, stable harmonic starlike, or stable
harmonic close-to-convex) in $\mathcal{U}$, if all functions $\mathfrak{f}%
_{\epsilon }=\mathfrak{s}+\epsilon \overline{\mathfrak{t}}$ with $|\epsilon
|=1$ are univalent (resp. convex, starlike, or close-to-convex) in $\mathcal{%
U}$. It is proved that $\mathfrak{f=s+}\overline{t}$ is stable harmonic
univalent (resp. convex, starlike, or close-to-convex) if and only if $%
F_{\epsilon }\mathfrak{=s+\epsilon t}$ are univalent (resp. convex,
starlike, or close-to-convex) in $\mathcal{U}$ for each $|\epsilon |=1.$

Recall that, convexity and starlikeness are hereditary properties for
conformal mappings and they do not extend to harmonic functions \cite{Duren}%
. The failure of hereditary properties leads to the notion of fully starlike
and fully convex functions which introduced by Chuaqui, Duren and Osgood 
\cite{Chuaqui}. A harmonic function $\mathfrak{f}$ of the unit disk is said
to be fully convex, if it maps every circle $\left \vert z\right \vert =r<1$
in a one-to-one manner onto a convex curve. Such a harmonic mapping $%
\mathfrak{f}$ with $\mathfrak{f}(0)=0$ is fully starlike if it maps every
circle $\left \vert z\right \vert =r<1$ in a one-to-one manner onto a curve
that bounds a domain starlike with respect to the origin. Denote by $%
\mathcal{FK}_{H}^{0}$ and $\mathcal{FS}_{H}^{\ast ,0}$ the subclasses of $%
\mathcal{K}_{H}^{0}$ and $\mathcal{S}_{H}^{\ast ,0}$ consisting of fully
convex and fully starlike functions, respectively. In \cite{Hernandez,
Nagpal-fully, Nagpaljkms}, it is proved that stable harmonic convex (or
stable harmonic starlike) mappings in $\mathcal{U}$ are fully convex (or
fully starlike) in $\mathcal{U}.$

In 2014, Nagpal and Ravichandran \cite{Nagpaljkms} studied a class $%
W_{H}^{0} $ of functions $\mathfrak{f}\in \mathcal{H}^{0}$ satisfying the
condition ${\text{Re}}\left[ \mathfrak{s}^{\prime }(z)+z\mathfrak{s}^{\prime
\prime }(z)\right] >\left \vert \mathfrak{t}^{\prime }(z)+z\mathfrak{t}%
^{\prime \prime }(z)\right \vert $ for $z\in \mathcal{U}$ which is harmonic
analogue of the class $W$ defined by Chichra \cite{Chichra} consisting of
functions $\mathfrak{f}\in \mathcal{A}$ satisfying the condition ${\text{Re}}%
\left[ \mathfrak{f}^{\prime }(z)+z\mathfrak{f}^{\prime \prime }(z)\right] >0$
for $z\in \mathcal{U}$ . It is stated that $W_{H}^{0}\subset \mathcal{S}%
_{H}^{\ast ,0}$ and in particular, the members of the class are fully
starlike in $\mathcal{U}.$

Ghosh and Vasudevarao \cite{Ghosh} investigated radius of convexity for the
partial sums of members of the class $W_{H}^{0}\left( \delta \right) $ of
functions $\mathfrak{f}\in \mathcal{H}^{0}$ satisfying the condition \newline
${\text{Re}}\left[ \mathfrak{s}^{\prime }(z)+\delta z\mathfrak{s}^{\prime
\prime }(z)\right] >\left \vert \mathfrak{t}^{\prime }(z)+\delta z\mathfrak{t%
}^{\prime \prime }(z)\right \vert $ for $\delta \geq 0,$ and $z\in \mathcal{U%
}. $

Further, Rajbala and Prajapat \cite{Rajbala} studied the class $%
W_{H}^{0}\left( \delta ,\lambda \right) $ of functions $\mathfrak{f}\in 
\mathcal{H}^{0}$ satisfying the condition ${\text{Re}}\left[ \mathfrak{s}%
^{\prime }(z)+\delta z\mathfrak{s}^{\prime \prime }(z)-\lambda \right]
>\left \vert \mathfrak{t}^{\prime }(z)+\delta z\mathfrak{t}^{\prime \prime
}(z)\right \vert $ for $\delta \geq 0,$ $0\leq \lambda <1,$ and $z\in 
\mathcal{U}.$ They constructed harmonic polynomials involving Gaussian
hypergeometric function which belong to the class $W_{H}^{0}\left( \delta
,\lambda \right) .$

Very recently, Ya\c{s}ar and Yal\c{c}\i n \cite{yasar} introduced the class $%
R_{H}^{0}\left( \delta ,\gamma \right) $ of functions $\mathfrak{f}\in 
\mathcal{H}^{0}$ satisfying the condition 
\begin{equation*}
{\text{Re}}\left[ \mathfrak{s}^{\prime }(z)+\delta z\mathfrak{s}^{\prime
\prime }(z)+\gamma z^{2}\mathfrak{s}^{\prime \prime \prime }\left( z\right) %
\right] >\left \vert \mathfrak{t}^{\prime }(z)+\delta z\mathfrak{t}^{\prime
\prime }(z)+\gamma z^{2}\mathfrak{t}^{\prime \prime \prime }\left( z\right)
\right \vert
\end{equation*}%
for $\delta \geq \gamma \geq 0,$ $z\in \mathcal{U}.$

In all studies mentioned above \cite{Nagpaljkms, Ghosh, Rajbala, yasar}, it
is proved that the functions in corresponding classes are close-to-convex.
Also, coefficient bounds, growth estimates, and convolution properties of
the classes are obtained.

\bigskip Denote by $\mathcal{R}_{H}^{0}(\gamma ,\delta ,\lambda ),$ the
class of functions $\mathfrak{f}=\mathfrak{s}+\overline{\mathfrak{t}}\in 
\mathcal{H}^{0}$ and satisfy%
\begin{equation}
{\text{Re}}\left[ \gamma \mathfrak{s}^{\prime }(z)+\delta z\mathfrak{s}%
^{\prime \prime }(z)+\left( \frac{\delta -\gamma }{2}\right) z^{2}\mathfrak{s%
}^{\prime \prime \prime }\left( z\right) -\lambda \right] >\left \vert
\gamma \mathfrak{t}^{\prime }(z)+\delta z\mathfrak{t}^{\prime \prime
}(z)+\left( \frac{\delta -\gamma }{2}\right) z^{2}\mathfrak{t}^{\prime
\prime \prime }\left( z\right) \right \vert  \label{BH}
\end{equation}%
where $0\leq \lambda <\gamma \leq \delta .$

It is evident that $W_{H}^{0}\equiv W_{H}^{0}\left( 1\right) \equiv \mathcal{%
R}_{H}^{0}(1,1,0),$ $W_{H}^{0}\left( 1,\lambda \right) \equiv \mathcal{R}%
_{H}^{0}(1,1,\lambda ),$ $R_{H}^{0}\left( \delta ,\frac{\delta -1}{2}\right)
\equiv \mathcal{R}_{H}^{0}(1,\delta ,0).$

Let $\mathcal{R}(\gamma ,\delta ,\lambda )$ denote a class of functions $%
\mathfrak{f}\in \mathcal{A}$ such that%
\begin{equation}
{\text{Re}}\left \{ \gamma \mathfrak{f}^{\prime }(z)+\delta z\mathfrak{f}%
^{\prime \prime }(z)+\left( \frac{\delta -\gamma }{2}\right) z^{2}\mathfrak{f%
}^{\prime \prime \prime }\left( z\right) \right \} >\lambda \text{ \ \ }%
\left( 0\leq \lambda <\gamma \leq \delta \right) .
\end{equation}%
The class $\mathcal{R}(\gamma ,\delta ,\lambda )$ is a particular case of
the class which is studied by Al-Refai \cite{Refai}. The starlikeness and
convexity of the class $\mathcal{R}(1,\delta ,\lambda )$ are studied in \cite%
{Rosihan2011, Rosihan2018}.

In this paper, we mainly deal with the functions $\mathfrak{f}=\mathfrak{s}+%
\overline{\mathfrak{t}}$ $\in \mathcal{H}^{0}$ of the class $\mathcal{R}%
_{H}^{0}(\gamma ,\delta ,\lambda )$ which is defined by the third-order
differential inequality (\ref{BH}). In the second section, we prove that the
members of the class $\mathcal{R}_{H}^{0}(\gamma ,\delta ,\lambda )$ are
close-to-convex. We also obtain coefficient bounds, growth estimates, and
sufficient coefficient condition of this class. In the third section, we
prove that this class is closed under convex combination and convolution of
its members. In the last section, we investigate the radii of fully
starlikeness and fully convexity of the class $\mathcal{R}_{H}^{0}(\gamma
,\delta ,\lambda ),$ and we give a result due to the class $\mathcal{R}%
_{H}^{0}(1,\delta ,\lambda )$ using previous works \cite{Rosihan2018} and 
\cite{Nagpaljkms}.

\section{Close-to-convexity, coefficient bounds, growth estimates}

First, we give a result of Clunie and Sheil-Small \cite{Clunie} which
derives a sufficient condition for $\mathfrak{f}\in \mathcal{H}$ to be
close-to-convex.

\begin{lemma}
Suppose $\mathfrak{s}$ and $\mathfrak{t}$ are analytic in $\mathcal{U}$ with 
$\left \vert \mathfrak{t}^{\prime }(0)\right \vert <\left \vert \mathfrak{s}%
^{\prime }(0)\right \vert $ and $F_{\epsilon }=\mathfrak{s}+\epsilon 
\mathfrak{t}$ is close-to-convex for each $\epsilon $ $\left( \left \vert
\epsilon \right \vert =1\right) ,$ then $\mathfrak{f}=\mathfrak{s}+\overline{%
\mathfrak{t}}$ is close-to-convex in $\mathcal{U}.$
\end{lemma}

\begin{theorem}
The harmonic mapping $\mathfrak{f}=\mathfrak{s}+\overline{\mathfrak{t}}\in 
\mathcal{R}_{H}^{0}(\gamma ,\delta ,\lambda )$ if and only if$\ F_{\epsilon
}=\mathfrak{s}+\epsilon \mathfrak{t}\in \mathcal{R}(\gamma ,\delta ,\lambda
) $ for each $\epsilon \left( \left \vert \epsilon \right \vert =1\right) .$
\end{theorem}

\begin{proof}
Suppose $\mathfrak{f}=\mathfrak{s}+\overline{\mathfrak{t}}$ $\in \mathcal{R}%
_{H}^{0}(\gamma ,\delta ,\lambda ).$ For each $\left \vert \epsilon
\right
\vert =1,$ 
\begin{align*}
& {\text{Re}}\left \{ \gamma F_{\epsilon }^{\prime }(z)+\delta zF_{\epsilon
}^{\prime \prime }(z)+\left( \frac{\delta -\gamma }{2}\right)
z^{2}F_{\epsilon }^{\prime \prime \prime }(z)\right \} \\
& ={\text{Re}}\left \{ \gamma \mathfrak{s}^{\prime }(z)+\delta z\mathfrak{s}%
^{\prime \prime }(z)+\left( \frac{\delta -\gamma }{2}\right) z^{2}\mathfrak{s%
}^{\prime \prime \prime }\left( z\right) \right. \\
& +\left. \epsilon \left( \gamma \mathfrak{t}^{\prime }(z)+\delta z\mathfrak{%
t}^{\prime \prime }(z)+\left( \frac{\delta -\gamma }{2}\right) z^{2}%
\mathfrak{t}^{\prime \prime \prime }\left( z\right) \right) \right \} \\
& >{\text{Re}}\left \{ \gamma \mathfrak{s}^{\prime }(z)+\delta z\mathfrak{s}%
^{\prime \prime }(z)+\left( \frac{\delta -\gamma }{2}\right) z^{2}\mathfrak{s%
}^{\prime \prime \prime }\left( z\right) \right \} \\
& -\left \vert \gamma \mathfrak{t}^{\prime }(z)+\delta z\mathfrak{t}^{\prime
\prime }(z)+\left( \frac{\delta -\gamma }{2}\right) z^{2}\mathfrak{t}%
^{\prime \prime \prime }\left( z\right) \right \vert \\
& >\lambda \text{ \ \ \ \ }\left( z\in \mathcal{U}\right) .
\end{align*}%
Thus, $F_{\epsilon }\in \mathcal{R}(\gamma ,\delta ,\lambda )$ for each $%
\epsilon \left( \left \vert \epsilon \right \vert =1\right) .$ Conversely,
let $F_{\epsilon }=\mathfrak{s}+\epsilon \mathfrak{t}\in \mathcal{R}(\gamma
,\delta ,\lambda )$ then%
\begin{eqnarray*}
&&{\text{Re}}\left \{ \gamma \mathfrak{s}^{\prime }(z)+\delta z\mathfrak{s}%
^{\prime \prime }(z)+\left( \frac{\delta -\gamma }{2}\right) z^{2}\mathfrak{s%
}^{\prime \prime \prime }\left( z\right) \right \} \\
&>&{\text{Re}}\left \{ -\epsilon \left( \gamma \mathfrak{t}^{\prime
}(z)+\delta z\mathfrak{t}^{\prime \prime }(z)+\left( \frac{\delta -\gamma }{2%
}\right) z^{2}\mathfrak{t}^{\prime \prime \prime }\left( z\right) \right)
\right \} +\lambda \text{ }\left( z\in \mathcal{U}\right) .
\end{eqnarray*}%
With appropriate choice of $\epsilon \left( \left \vert \epsilon
\right
\vert =1\right) ,$ it follows that%
\begin{eqnarray*}
&&{\text{Re}}\left \{ \gamma \mathfrak{s}^{\prime }(z)+\delta z\mathfrak{s}%
^{\prime \prime }(z)+\left( \frac{\delta -\gamma }{2}\right) z^{2}\mathfrak{s%
}^{\prime \prime \prime }\left( z\right) -\lambda \right \} \\
&>&\left \vert \gamma \mathfrak{t}^{\prime }(z)+\delta z\mathfrak{t}^{\prime
\prime }(z)+\left( \frac{\delta -\gamma }{2}\right) z^{2}\mathfrak{t}%
^{\prime \prime \prime }\left( z\right) \right \vert \text{ }\left( z\in 
\mathcal{U}\right) ,
\end{eqnarray*}%
and hence $\mathfrak{f}\in \mathcal{R}_{H}^{0}(\gamma ,\delta ,\lambda ).$
\end{proof}

\begin{lemma}
(Jack-Miller-Mocanu Lemma \cite{Jack-Miller lemma, Miller}) Let $w$ defined
by $w(z)=c_{n}z^{n}+c_{n+1}z^{n+1}+...$ be analytic in $\mathcal{U},$ with $%
c_{n}\neq 0,$ and let $z_{0}\neq 0,~z_{0}=r_{0}e^{i\theta _{0}}(0<r_{0}<1)$
be a point of $\mathcal{U}$ such that 
\begin{equation*}
|w(z_{0})|=\max_{|z|\leq |z_{0}|}|w(z)|
\end{equation*}%
then there is a real number $k,$ $k\geq n\geq 1,$ such that 
\begin{equation*}
\frac{z_{0}w^{\prime }(z_{0})}{w(z_{0})}=k\text{ \ and \ }{\text{Re}}\left
\{ 1+\frac{z_{0}w^{\prime \prime }(z_{0})}{w^{\prime }(z_{0})}\right \} \geq
k.
\end{equation*}
\end{lemma}

\begin{lemma}
If $F\in \mathcal{R}(\gamma ,\delta ,\lambda )$ then ${\text{Re}}\{F^{\prime
}(z)\}>0,$ and hence $F$ is close-to-convex in $\mathcal{U}.$
\end{lemma}

\begin{proof}
Suppose $F\in \mathcal{R}(\gamma ,\delta ,\lambda )$ and $\frac{2\gamma
F^{\prime }(z)+2\delta zF^{\prime \prime }(z)+\left( \delta -\gamma \right)
z^{2}F^{\prime \prime \prime }(z)-2\lambda }{2\left( \gamma -\lambda \right) 
}=:\Psi (z).$ Then ${\text{Re}}\{ \Psi (z)\}>0$ for $z\in \mathcal{U}.$
Consider an analytic function $w$ in $\mathcal{U}$ with $w(0)=0$ and%
\begin{equation*}
F^{\prime }(z)=\frac{1+w(z)}{1-w(z)},~\ w(z)\neq 1.
\end{equation*}%
We need to prove that $|w(z)|<1$ for all $z\in \mathcal{U}.$ Then we have%
\begin{eqnarray*}
\Psi (z) &=&\frac{2\gamma F^{\prime }(z)+2\delta zF^{\prime \prime
}(z)+\left( \delta -\gamma \right) z^{2}F^{\prime \prime \prime
}(z)-2\lambda }{2\left( \gamma -\lambda \right) } \\
&=&\frac{\gamma }{\gamma -\lambda }\frac{1+w(z)}{1-w(z)}+\frac{2\delta }{%
\gamma -\lambda }\frac{zw^{\prime }(z)}{(1-w(z))^{2}} \\
&&+\frac{\delta -\gamma }{\gamma -\lambda }\frac{z^{2}\left[ w^{\prime
\prime }\left( z\right) \left( 1-w(z)\right) +2\left( w^{\prime }\left(
z\right) \right) ^{2}\right] }{(1-w(z))^{3}}-\frac{\lambda }{\gamma -\lambda 
} \\
&=&\frac{1}{\gamma -\lambda }\left( \gamma \frac{1+w(z)}{1-w(z)}+2\delta 
\frac{zw^{\prime }(z)}{(1-w(z))^{2}}\right. \\
&&\left. +\left( \delta -\gamma \right) \frac{zw^{\prime }(z)}{(1-w(z))^{2}}%
\frac{zw^{\prime \prime }\left( z\right) }{w^{\prime }(z)}+2\left( \delta
-\gamma \right) \frac{\left( zw^{\prime }(z)\right) ^{2}}{(1-w(z))^{3}}%
-\lambda \right) .
\end{eqnarray*}%
Since $w$ is analytic in $\mathcal{U}$ and $w(0)=0,$ if there are $z_{0}\in 
\mathcal{U}$ such that 
\begin{equation*}
\max_{|z|\leq |z_{0}|}|w(z)|=|w(z_{0})|=1,
\end{equation*}%
then by Lemma 3, we can write 
\begin{equation*}
w(z_{0})=e^{i\theta },\quad z_{0}w^{\prime }(z_{0})=kw(z_{0})=ke^{i\theta
},\quad (k\geq 1,\text{ }0<\theta <2\pi ).
\end{equation*}%
and%
\begin{equation*}
\text{Re}\left \{ \frac{z_{0}w^{\prime \prime }(z_{0})}{w^{\prime }(z_{0})}%
\right \} \geq k-1.
\end{equation*}%
For such a point $z_{0}\in \mathcal{U}$, we obtain%
\begin{eqnarray*}
{\text{Re}}\{ \Psi (z_{0})\} &=&\frac{1}{\gamma -\lambda }\text{Re}\left(
\gamma \frac{1+w(z_{0})}{1-w(z_{0})}+2\delta \frac{z_{0}w^{\prime }(z_{0})}{%
(1-w(z_{0}))^{2}}\right. \\
&&\left. +\left( \delta -\gamma \right) \frac{z_{0}w^{\prime }(z_{0})}{%
(1-w(z_{0}))^{2}}\frac{z_{0}w^{\prime \prime }\left( z_{0}\right) }{%
w^{\prime }(z_{0})}+2\left( \delta -\gamma \right) \frac{\left(
z_{0}w^{\prime }(z_{0})\right) ^{2}}{(1-w(z_{0}))^{3}}-\lambda \right) \\
&=&\frac{1}{\gamma -\lambda }\left[ -\frac{\delta k}{1-\cos \theta }-\frac{%
\left( \delta -\gamma \right) k}{2\left( 1-\cos \theta \right) }\text{Re}%
\left \{ \frac{zw^{\prime \prime }\left( z_{0}\right) }{w^{\prime }(z_{0})}%
\right \} +\frac{\left( \delta -\gamma \right) k^{2}}{2\left( 1-\cos \theta
\right) }-\lambda \right] \\
&\leq &\frac{1}{\gamma -\lambda }\left[ -\frac{\delta k}{1-\cos \theta }+%
\frac{\left( \delta -\gamma \right) k}{2\left( 1-\cos \theta \right) }\left(
1-k\right) +\frac{\left( \delta -\gamma \right) k^{2}}{2\left( 1-\cos \theta
\right) }-\lambda \right] \\
&=&-\frac{1}{\gamma -\lambda }\left[ \frac{\left( \delta +\gamma \right) k}{%
2\left( 1-\cos \theta \right) }+\lambda \right] <0,
\end{eqnarray*}%
which contradicts our assumption. Hence, there is no $z_{0}\in \mathcal{U}$
such that $|w(z_{0})|=1,$ which means that $|w(z)|<1$ for all $z\in \mathcal{%
U}.$ Therefore, we obtain that ${\text{Re}}\{F^{\prime }(z)\}>0.$
\end{proof}

\begin{theorem}
The functions in the class $\mathcal{R}_{H}^{0}(\gamma ,\delta ,\lambda )$
are close-to-convex in $\mathcal{U}.$
\end{theorem}

\begin{proof}
Referring to Lemma 4, we derive that functions $F_{\epsilon }=\mathfrak{s}%
+\epsilon \mathfrak{t}\in \mathcal{R}(\gamma ,\delta ,\lambda )$ are
close-to-convex in $\mathcal{U}$ for each $\epsilon (|\epsilon |=1).$ Now in
view of Lemma 1 and Theorem 2, we obtain that functions in $\mathcal{R}%
_{H}^{0}(\gamma ,\delta ,\lambda )$ are close-to-convex in $\mathcal{U}.$
\end{proof}

\begin{theorem}
Let $\mathfrak{f}=\mathfrak{s}+\overline{\mathfrak{t}}\in \mathcal{R}%
_{H}^{0}(\gamma ,\delta ,\lambda )$ then for $m\geq 2,$%
\begin{equation}
\left \vert b_{m}\right \vert \leq \frac{2\left( \gamma -\lambda \right) }{%
m^{2}\left[ 2\gamma +\left( \delta -\gamma \right) (m-1)\right] }.
\label{boundbn}
\end{equation}%
The result is sharp and equality holds for the function $\mathfrak{f}(z)=z+%
\frac{2\left( \gamma -\lambda \right) }{m^{2}\left[ 2\gamma +\left( \delta
-\gamma \right) (m-1)\right] }\bar{z}^{m}.$
\end{theorem}

\begin{proof}
Suppose that $\mathfrak{f}=\mathfrak{s}+\overline{\mathfrak{t}}\in \mathcal{R%
}_{H}^{0}(\gamma ,\delta ,\lambda ).$ Using the series representation of $%
\mathfrak{t}(z),$ we derive%
\begin{eqnarray*}
&&r^{m-1}m^{2}\left[ \gamma +\frac{\delta -\gamma }{2}(m-1)\right] \left
\vert b_{m}\right \vert \\
&\leq &\frac{1}{2\pi }\int \limits_{0}^{2\pi }\left \vert \gamma \mathfrak{t}%
^{\prime }(re^{i\theta })+\delta re^{i\theta }\mathfrak{t}^{\prime \prime
}(re^{i\theta })+\left( \frac{\delta -\gamma }{2}\right) r^{2}e^{2i\theta }%
\mathfrak{t}^{\prime \prime \prime }(re^{i\theta })\right \vert d\theta \\
&<&\frac{1}{2\pi }\int \limits_{0}^{2\pi }{\text{Re}}\left \{ \gamma 
\mathfrak{s}^{\prime }(re^{i\theta })+\delta re^{i\theta }\mathfrak{s}%
^{\prime \prime }(re^{i\theta })+\left( \frac{\delta -\gamma }{2}\right)
r^{2}e^{2i\theta }\mathfrak{s}^{\prime \prime \prime }(re^{i\theta
})-\lambda \right \} d\theta \\
&=&\frac{1}{2\pi }\int \limits_{0}^{2\pi }{\text{Re}}\left \{ \gamma
-\lambda +\sum \limits_{m=2}^{\infty }m^{2}\left[ \gamma +\frac{\delta
-\gamma }{2}(m-1)\right] a_{m}r^{m-1}e^{i(m-1)\theta }\right \} d\theta \\
&=&\gamma -\lambda .
\end{eqnarray*}%
Allowing $r\rightarrow 1^{-}$ gives the desired bound. Moreover, it is easy
to verify that the equality holds for the function $\mathfrak{f}(z)=z+\frac{%
2\left( \gamma -\lambda \right) }{m^{2}\left[ 2\gamma +\left( \delta -\gamma
\right) (m-1)\right] }\bar{z}^{m}.$
\end{proof}

\begin{theorem}
Let $\mathfrak{f}=\mathfrak{s}+\overline{\mathfrak{t}}\in \mathcal{R}%
_{H}^{0}(\gamma ,\delta ,\lambda )$. Then for $m\geq 2,$ we have%
\begin{align*}
\text{ \ \ \ \ \ \ \ }(i)\text{ }\left \vert a_{m}\right \vert +\left \vert
b_{m}\right \vert & \leq \frac{4\left( \gamma -\lambda \right) }{m^{2}\left[
2\gamma +\left( \delta -\gamma \right) (m-1)\right] }, \\
\text{ \ \ \ \ \ \ }(ii)\text{ }\left \vert \left \vert a_{m}\right \vert
-\left \vert b_{m}\right \vert \right \vert & \leq \frac{4\left( \gamma
-\lambda \right) }{m^{2}\left[ 2\gamma +\left( \delta -\gamma \right) (m-1)%
\right] }, \\
(iii)\text{ }\left \vert a_{m}\right \vert & \leq \frac{4\left( \gamma
-\lambda \right) }{m^{2}\left[ 2\gamma +\left( \delta -\gamma \right) (m-1)%
\right] }.
\end{align*}%
All these results are sharp and all equalities hold for the function $%
\mathfrak{f}(z)=z+\sum \limits_{m=2}^{\infty }\frac{4\left( \gamma -\lambda
\right) }{m^{2}\left[ 2\gamma +\left( \delta -\gamma \right) (m-1)\right] }%
z^{m}.$
\end{theorem}

\begin{proof}
$(i)$ Suppose that $\mathfrak{f}=\mathfrak{s}+\overline{\mathfrak{t}}\in 
\mathcal{R}_{H}^{0}(\gamma ,\delta ,\lambda ),$ then from Theorem 2, $%
F_{\epsilon }=\mathfrak{s}+\epsilon \mathfrak{t}\in \mathcal{R}(\gamma
,\delta ,\lambda )\ $for $\epsilon $ $\left( \left \vert \epsilon
\right
\vert =1\right) .$ Thus for each $\left \vert \epsilon \right \vert
=1,$ we have 
\begin{equation*}
\text{Re}\left \{ \gamma (\mathfrak{s}+\epsilon \mathfrak{t})^{\prime
}+\delta z(\mathfrak{s}+\epsilon \mathfrak{t})^{\prime \prime }+\left( \frac{%
\delta -\gamma }{2}\right) z^{2}(\mathfrak{s}+\epsilon \mathfrak{t})^{\prime
\prime \prime }\right \} >\lambda
\end{equation*}%
for $z\in \mathcal{U}.$ This implies that there exists an analytic function $%
p$ of the form $p(z)=1+\underset{m=1}{\overset{\infty }{\sum }}p_{m}z^{m},$
with Re$\left[ p(z)\right] >0$ in $\mathcal{U}$ such that 
\begin{align}
& \gamma \mathfrak{s}^{\prime }(z)+\delta z\mathfrak{s}^{\prime \prime
}(z)+\left( \frac{\delta -\gamma }{2}\right) z^{2}\mathfrak{s}^{\prime
\prime \prime }\left( z\right) +\epsilon \left( \gamma \mathfrak{t}^{\prime
}(z)+\delta z\mathfrak{t}^{\prime \prime }(z)+\left( \frac{\delta -\gamma }{2%
}\right) z^{2}\mathfrak{t}^{\prime \prime \prime }\left( z\right) \right) 
\notag \\
& =\lambda +\left( \gamma -\lambda \right) p(z).  \label{eqp}
\end{align}%
Comparing coefficients on both sides of (\ref{eqp}) we have%
\begin{equation*}
m^{2}\left[ \gamma +\frac{\delta -\gamma }{2}(m-1)\right] (a_{m}+\epsilon
b_{m})=\left( \gamma -\lambda \right) p_{m-1}\text{ for }m\geq 2.
\end{equation*}%
Since $\left \vert p_{m}\right \vert \leq 2$ for $m\geq 1,$ and $\epsilon
\left( \left \vert \epsilon \right \vert =1\right) $ is arbitrary, proof of
(i) is complete. Proofs of (ii) and (iii) follows from (i). The function $%
\mathfrak{f}(z)=z+$ $\sum \limits_{m=2}^{\infty }\frac{4\left( \gamma
-\lambda \right) }{m^{2}\left[ 2\gamma +\left( \delta -\gamma \right) (m-1)%
\right] }z^{m}$, shows that all inequalities are sharp.
\end{proof}

The following result gives a sufficient condition for a function to be in
the class $\mathcal{R}_{H}^{0}(\gamma ,\delta ,\lambda ).$

\begin{theorem}
Let $\mathfrak{f}=\mathfrak{s}+\overline{\mathfrak{t}}\in \mathcal{H}^{0}$
with

\begin{equation}
\sum \limits_{m=2}^{\infty }m^{2}\left[ 2\gamma +\left( \delta -\gamma
\right) (m-1)\right] \left( \left \vert a_{m}\right \vert +\left \vert
b_{m}\right \vert \right) \leq 2\left( \gamma -\lambda \right) ,
\label{coeffcond}
\end{equation}%
then $\mathfrak{f}\in \mathcal{R}_{H}^{0}(\gamma ,\delta ,\lambda ).$
\end{theorem}

\begin{proof}
Suppose that $\mathfrak{f}=\mathfrak{s}+\overline{\mathfrak{t}}\in \mathcal{H%
}^{0}.$ Then using (\ref{coeffcond}),

\begin{eqnarray*}
&&{\text{Re}}\left \{ \gamma \mathfrak{s}^{\prime }(z)+\delta z\mathfrak{s}%
^{\prime \prime }(z)+\left( \frac{\delta -\gamma }{2}\right) z^{2}\mathfrak{s%
}^{\prime \prime \prime }\left( z\right) -\lambda \right \} \\
&=&{\text{Re}}\left \{ \gamma -\lambda +\sum \limits_{m=2}^{\infty }m^{2}%
\left[ \gamma +\frac{\delta -\gamma }{2}(m-1)\right] a_{m}z^{m-1}\right \} \\
&>&\gamma -\lambda -\sum \limits_{m=2}^{\infty }m^{2}\left[ \gamma +\frac{%
\delta -\gamma }{2}(m-1)\right] \left \vert a_{m}\right \vert \\
&\geq &\sum \limits_{m=2}^{\infty }m^{2}\left[ \gamma +\frac{\delta -\gamma 
}{2}(m-1)\right] \left \vert b_{m}\right \vert \\
&>&\left \vert \sum \limits_{m=2}^{\infty }m^{2}\left[ \gamma +\frac{\delta
-\gamma }{2}(m-1)\right] b_{m}z^{m-1}\right \vert \\
&=&\left \vert \gamma \mathfrak{t}^{\prime }(z)+\delta z\mathfrak{t}^{\prime
\prime }(z)+\left( \frac{\delta -\gamma }{2}\right) z^{2}\mathfrak{t}%
^{\prime \prime \prime }\left( z\right) \right \vert .
\end{eqnarray*}%
Hence, $\mathfrak{f}\in \mathcal{R}_{H}^{0}(\gamma ,\delta ,\lambda ).$
\end{proof}

\begin{corollary}
Let $\mathfrak{f}=\mathfrak{s}+\overline{\mathfrak{t}}\in \mathcal{H}^{0}$
satisfies the inequality (\ref{coeffcond}), then f is stable harmonic
close-to-convex in $\mathcal{U}$.
\end{corollary}

\begin{theorem}
\label{theodist} Let $\mathfrak{f}=\mathfrak{s}+\overline{\mathfrak{t}}$ $%
\in \mathcal{R}_{H}^{0}(\gamma ,\delta ,\lambda ).$ Then%
\begin{equation*}
\left \vert z\right \vert +4\left( \gamma -\lambda \right) \sum
\limits_{m=2}^{\infty }\frac{(-1)^{m-1}\left \vert z\right \vert ^{m}}{m^{2}%
\left[ 2\gamma +\left( \delta -\gamma \right) (m-1)\right] }\leq \left \vert 
\mathfrak{f}(z)\right \vert ,
\end{equation*}%
\begin{equation*}
\left \vert \mathfrak{f}(z)\right \vert \leq \left \vert z\right \vert
+4\left( \gamma -\lambda \right) \sum \limits_{m=2}^{\infty }\frac{\left
\vert z\right \vert ^{m}}{m^{2}\left[ 2\gamma +\left( \delta -\gamma \right)
(m-1)\right] }.
\end{equation*}%
Inequalities are sharp for the function $f(z)=z+\sum \limits_{m=2}^{\infty }%
\frac{4\left( \gamma -\lambda \right) }{m^{2}\left[ 2\gamma +\left( \delta
-\gamma \right) (m-1)\right] }\overline{z}^{m}.$
\end{theorem}

\begin{proof}
Let $\mathfrak{f}=\mathfrak{s}+\overline{\mathfrak{t}}$ $\in \mathcal{R}%
_{H}^{0}(\gamma ,\delta ,\lambda ).$ Then using Theorem 2\textbf{,} $%
F_{\epsilon }\in \mathcal{R}(\gamma ,\delta ,\lambda )\ $and for each $%
\left
\vert \epsilon \right \vert =1$ we have $\text{Re}\left \{ \psi
(z)\right \} >\lambda $ where%
\begin{equation*}
\psi (z)=\gamma F_{\epsilon }^{\prime }(z)+\delta zF_{\epsilon }^{\prime
\prime }(z)+\frac{\delta -\gamma }{2}z^{2}F_{\epsilon }^{\prime \prime
\prime }(z).
\end{equation*}%
Then, we have 
\begin{eqnarray*}
\psi (z) &=&\left( \frac{\delta -\gamma }{2}\right) \left[ \frac{2\gamma }{%
\delta -\gamma }F_{\epsilon }^{\prime }(z)+\left( \frac{2\gamma }{\delta
-\gamma }+2\right) zF_{\epsilon }^{\prime \prime }(z)+z^{2}F_{\epsilon
}^{\prime \prime \prime }(z)\right] \\
&=&\left( \frac{\delta -\gamma }{2}\right) \left[ \frac{2\gamma }{\delta
-\gamma }\left( zF_{\epsilon }^{\prime }(z)\right) ^{\prime }+\left(
z^{2}F_{\epsilon }^{\prime \prime }(z)\right) ^{\prime }\right] \\
&=&\left( \frac{\delta -\gamma }{2}\right) \left[ \frac{2\gamma }{\delta
-\gamma }\left( zF_{\epsilon }^{\prime }(z)\right) +\left( z^{2}F_{\epsilon
}^{\prime \prime }(z)\right) \right] ^{\prime } \\
&=&\left( \frac{\delta -\gamma }{2}\right) \left[ z^{2-\frac{2\gamma }{%
\delta -\gamma }}\left( z^{\frac{2\gamma }{\delta -\gamma }}F_{\epsilon
}^{\prime }(z)\right) ^{\prime }\right] ^{\prime }.
\end{eqnarray*}%
Then integrating from $0$ to $z$ gives%
\begin{equation*}
\left( \frac{2}{\delta -\gamma }\right) z^{\frac{2\gamma }{\delta -\gamma }%
-2}\int \limits_{0}^{z}\psi (\omega )d\omega =\left( z^{\frac{2\gamma }{%
\delta -\gamma }}F_{\epsilon }^{\prime }(z)\right) ^{\prime }.
\end{equation*}%
Making the substitution $\omega =r^{\frac{\delta -\gamma }{2}}z$ in the
above integral and integrating again, change of variables gives%
\begin{equation}
F_{\epsilon }^{\prime }(z)=\frac{1}{\gamma }\int \limits_{0}^{1}\int
\limits_{0}^{1}\psi (v^{\frac{\delta -\gamma }{2\gamma }}uz)dudv.
\label{doubleint}
\end{equation}%
On the other hand, since Re$\left \{ \frac{\psi (z)-\lambda }{\gamma
-\lambda }\right \} >0$ then $\psi (z)\prec $ $\frac{\gamma +\left( \gamma
-2\lambda \right) z}{1-z}$ where $\prec $ denotes the subordination \cite%
{Duren}. Let 
\begin{equation*}
\phi (z)=\int \limits_{0}^{1}\int \limits_{0}^{1}\frac{dudv}{1-uv^{\frac{%
\delta -\gamma }{2\gamma }}z}=1+\sum \limits_{m=1}^{\infty }\frac{z^{m}}{%
(1+m)\left( 1+\frac{\delta -\gamma }{2\gamma }m\right) }
\end{equation*}%
and 
\begin{equation*}
h(z)=\frac{1}{\gamma }\left( \frac{\gamma +\left( \gamma -2\lambda \right) z%
}{1-z}\right) =1+\sum \limits_{m=1}^{\infty }\frac{2\left( \gamma -\lambda
\right) }{\gamma }z^{m}.
\end{equation*}%
Then, from (\ref{doubleint}) we have 
\begin{eqnarray*}
F_{\epsilon }^{\prime }(z) &\prec &(\phi \ast h)(z) \\
&=&\left( 1+\sum \limits_{m=1}^{\infty }\frac{z^{m}}{(1+m)\left( 1+\frac{%
\delta -\gamma }{2\gamma }m\right) }\right) \ast \left( 1+\sum
\limits_{m=1}^{\infty }\frac{2\left( \gamma -\lambda \right) }{\gamma }%
z^{m}\right) \\
&=&1+\sum \limits_{m=1}^{\infty }\frac{4\left( \gamma -\lambda \right) }{%
m^{2}\left( \delta -\gamma \right) +m\left( \delta +\gamma \right) +2\gamma }%
z^{m}.
\end{eqnarray*}%
Since 
\begin{eqnarray*}
\left \vert F_{\epsilon }^{\prime }(z)\right \vert &=&\left \vert \mathfrak{s%
}^{\prime }(z)+\epsilon \mathfrak{t}^{\prime }(z)\right \vert \\
&\leq &1+4\left( \gamma -\lambda \right) \sum \limits_{m=1}^{\infty }\frac{%
\left \vert z\right \vert ^{m}}{m^{2}\left( \delta -\gamma \right) +m\left(
\delta +\gamma \right) +2\gamma }
\end{eqnarray*}%
and%
\begin{eqnarray*}
\left \vert F_{\epsilon }^{\prime }(z)\right \vert &=&\left \vert \mathfrak{s%
}^{\prime }(z)+\epsilon \mathfrak{t}^{\prime }(z)\right \vert \\
&\geq &1+4\left( \gamma -\lambda \right) \sum \limits_{m=1}^{\infty }\frac{%
(-1)^{m}\left \vert z\right \vert ^{m}}{m^{2}\left( \delta -\gamma \right)
+m\left( \delta +\gamma \right) +2\gamma },
\end{eqnarray*}%
in particular we have 
\begin{equation*}
\left \vert \mathfrak{s}^{\prime }(z)\right \vert +\left \vert \mathfrak{t}%
^{\prime }(z)\right \vert \leq 1+4\left( \gamma -\lambda \right) \sum
\limits_{m=1}^{\infty }\frac{\left \vert z\right \vert ^{m}}{m^{2}\left(
\delta -\gamma \right) +m\left( \delta +\gamma \right) +2\gamma }
\end{equation*}%
and%
\begin{equation*}
\left \vert \mathfrak{s}^{\prime }(z)\right \vert -\left \vert \mathfrak{t}%
^{\prime }(z)\right \vert \geq 1+4\left( \gamma -\lambda \right) \sum
\limits_{m=1}^{\infty }\frac{(-1)^{m}\left \vert z\right \vert ^{m}}{%
m^{2}\left( \delta -\gamma \right) +m\left( \delta +\gamma \right) +2\gamma }%
.
\end{equation*}%
Let $\Gamma $ be the radial segment from $0$ to $z,$ then%
\begin{eqnarray*}
\left \vert \mathfrak{f}(z)\right \vert &=&\left \vert \int \limits_{\Gamma }%
\frac{\partial \mathfrak{f}}{\partial \zeta }d\zeta +\frac{\partial 
\mathfrak{f}}{\partial \bar{\zeta}}d\bar{\zeta}\right \vert \leq \int
\limits_{\Gamma }\left( \left \vert \mathfrak{s}^{\prime }(\zeta )\right
\vert +\left \vert \mathfrak{t}^{\prime }(\zeta )\right \vert \right) \left
\vert d\zeta \right \vert \\
&\leq &\int \limits_{0}^{\left \vert z\right \vert }\left( 1+4\left( \gamma
-\lambda \right) \sum \limits_{m=1}^{\infty }\frac{\left \vert \tau \right
\vert ^{m}}{m^{2}\left( \delta -\gamma \right) +m\left( \delta +\gamma
\right) +2\gamma }\right) d\tau \\
&=&\left \vert z\right \vert +4\left( \gamma -\lambda \right) \sum
\limits_{m=1}^{\infty }\frac{\left \vert z\right \vert ^{m+1}}{(m+1)\left[
m^{2}\left( \delta -\gamma \right) +m\left( \delta +\gamma \right) +2\gamma %
\right] } \\
&=&\left \vert z\right \vert +4\left( \gamma -\lambda \right) \sum
\limits_{m=2}^{\infty }\frac{\left \vert z\right \vert ^{m}}{m\left[ \left(
m-1\right) ^{2}\left( \delta -\gamma \right) +\left( m-1\right) \left(
\delta +\gamma \right) +2\gamma \right] } \\
&=&\left \vert z\right \vert +4\left( \gamma -\lambda \right) \sum
\limits_{m=2}^{\infty }\frac{\left \vert z\right \vert ^{m}}{m^{2}\left[
2\gamma +\left( \delta -\gamma \right) \left( m-1\right) \right] }
\end{eqnarray*}%
and 
\begin{eqnarray*}
\left \vert \mathfrak{f}(z)\right \vert &\geq &\int \limits_{\Gamma }\left(
\left \vert \mathfrak{s}^{\prime }(\zeta )\right \vert -\left \vert 
\mathfrak{t}^{\prime }(\zeta )\right \vert \right) \left \vert d\zeta \right
\vert \\
&\geq &\int \limits_{0}^{\left \vert z\right \vert }\left( 1+4\left( \gamma
-\lambda \right) \sum \limits_{m=1}^{\infty }\frac{(-1)^{m}\left \vert \tau
\right \vert ^{m}}{m^{2}\left( \delta -\gamma \right) +m\left( \delta
+\gamma \right) +2\gamma }\right) d\tau \\
&=&\left \vert z\right \vert +4\left( \gamma -\lambda \right) \sum
\limits_{m=2}^{\infty }\frac{(-1)^{m-1}\left \vert z\right \vert ^{m}}{m^{2}%
\left[ 2\gamma +\left( \delta -\gamma \right) \left( m-1\right) \right] }.
\end{eqnarray*}
\end{proof}

\section{Convex Combinations and Convolutions}

In this section, we prove that the class $\mathcal{R}_{H}^{0}(\gamma ,\delta
,\lambda )$ is closed under convex combinations and convolutions of its
members.

\begin{theorem}
The class $\mathcal{R}_{H}^{0}(\gamma ,\delta ,\lambda )$ is closed under
convex combinations.
\end{theorem}

\begin{proof}
Suppose $\mathfrak{f}_{i}=\mathfrak{s}_{i}+\overline{\mathfrak{t}_{i}}\in 
\mathcal{R}_{H}^{0}(\gamma ,\delta ,\lambda )$ for $i=1,2,...,n$ and $\sum
\limits_{i=1}^{n}\varrho _{i}=1$ $(0\leq \varrho _{i}\leq 1).$ The convex
combination of functions $\mathfrak{f}_{i}$ $\left( i=1,2,...,n\right) $ may
be written as%
\begin{equation*}
\mathfrak{f}(z)=\sum \limits_{i=1}^{n}\varrho _{i}\mathfrak{f}_{i}(z)=%
\mathfrak{s}(z)+\overline{\mathfrak{t}(z)},
\end{equation*}%
where 
\begin{equation*}
\mathfrak{s}(z)=\sum \limits_{i=1}^{n}\varrho _{i}\mathfrak{s}_{i}\left(
z\right) \text{ \ and \ }\mathfrak{t}(z)=\sum \limits_{i=1}^{n}\varrho _{i}%
\mathfrak{t}_{i}\left( z\right) .
\end{equation*}%
Then both $\mathfrak{s}$ and $\mathfrak{t}$ are analytic in $\mathcal{U}$
with $\mathfrak{s}(0)=\mathfrak{t}(0)=\mathfrak{s}^{\prime }(0)-1=\mathfrak{t%
}^{\prime }(0)=0$ and%
\begin{eqnarray*}
&&{\text{Re}}\{ \gamma \mathfrak{s}^{\prime }(z)+\delta z\mathfrak{s}%
^{\prime \prime }(z)+\left( \frac{\delta -\gamma }{2}\right) z^{2}\mathfrak{s%
}^{\prime \prime \prime }\left( z\right) -\lambda \} \\
&=&{\text{Re}}\left \{ \sum \limits_{i=1}^{n}\varrho _{i}\left( \gamma 
\mathfrak{s}_{i}^{\prime }(z)+\delta z\mathfrak{s}_{i}^{\prime \prime
}(z)+\left( \frac{\delta -\gamma }{2}\right) z^{2}\mathfrak{s}_{i}^{\prime
\prime \prime }\left( z\right) -\lambda \right) \right \} \\
&>&\sum \limits_{i=1}^{n}\varrho _{i}\left \vert \gamma \mathfrak{t}%
_{i}^{\prime }(z)+\delta z\mathfrak{t}_{i}^{\prime \prime }(z)+\left( \frac{%
\delta -\gamma }{2}\right) z^{2}\mathfrak{t}_{i}^{\prime \prime \prime
}\left( z\right) \right \vert \\
&\geq &\left \vert \gamma \mathfrak{t}^{\prime }(z)+\delta z\mathfrak{t}%
^{\prime \prime }(z)+\left( \frac{\delta -\gamma }{2}\right) z^{2}\mathfrak{t%
}^{\prime \prime \prime }\left( z\right) \right \vert
\end{eqnarray*}%
showing that $\mathfrak{f}\in \mathcal{R}_{H}^{0}(\gamma ,\delta ,\lambda )$.
\end{proof}

A sequence $\{c_{m}\}_{m=0}^{\infty }$ of non-negative real numbers is said
to be a convex null sequence, if $c_{m}\rightarrow 0$ as $m\rightarrow
\infty $, and $c_{0}-c_{1}\geq c_{1}-c_{2}\geq c_{2}-c_{3}\geq ...\geq
c_{m-1}-c_{m}\geq ...\geq 0.$ To prove results for convolution, we shall
need the following Lemma 12 and Lemma 13.

\begin{lemma}
\cite{Fejer} If $\{c_{m}\}_{m=0}^{\infty }$ be a convex null sequence, then
function 
\begin{equation*}
q(z)=\dfrac{c_{0}}{2}+\sum_{m=1}^{\infty }c_{m}z^{m}
\end{equation*}%
is analytic and ${\text{Re}}\{q(z)\}>0$ in $\mathcal{U}.$
\end{lemma}

\begin{lemma}
\cite{Singh} Let the function $p$ be analytic in $\mathcal{U}$ with $p(0)=1$
and ${\text{Re}}\{p(z)\}>1/2$ in $\mathcal{U}.$ Then for any analytic
function $F$ in $\mathcal{U},$ the function $p\ast F$ takes values in the
convex hull of the image of $\mathcal{U}$ under $F.$
\end{lemma}

\begin{lemma}
Let $F\in \mathcal{R}(\gamma ,\delta ,\lambda ),$ then ${\text{Re}}\left \{ 
\dfrac{F(z)}{z}\right \} >\dfrac{1}{2}.$
\end{lemma}

\begin{proof}
Suppose $F\in \mathcal{R}(\gamma ,\delta ,\lambda )$ be given by $%
F(z)=z+\sum_{m=2}^{\infty }A_{m}z^{m},$ then%
\begin{equation*}
{\text{Re}}\left \{ \gamma +\sum_{m=2}^{\infty }m^{2}\left[ \gamma +\frac{%
\delta -\gamma }{2}(m-1)\right] A_{m}z^{m-1}\right \} >\lambda \quad (z\in 
\mathcal{U}),
\end{equation*}%
which is equivalent to ${\text{Re}}\{p(z)\}>\frac{1}{2}$ in $\mathcal{U},$
where 
\begin{equation*}
p(z)=1+\frac{1}{4\left( \gamma -\lambda \right) }\sum_{m=2}^{\infty }m^{2}%
\left[ 2\gamma +\left( \delta -\gamma \right) \left( m-1\right) \right]
A_{m}z^{m-1}.
\end{equation*}%
Now consider a sequence $\{c_{m}\}_{m=0}^{\infty }$ defined by 
\begin{equation*}
c_{0}=1\text{ and }c_{m-1}=\dfrac{4\left( \gamma -\lambda \right) }{m^{2}%
\left[ 2\gamma +\left( \delta -\gamma \right) (m-1)\right] }\text{ \ for \ }%
m\geq 2.
\end{equation*}%
It can be easily seen that the sequence $\{c_{m}\}_{m=0}^{\infty }$ is a
convex null sequence. Using Lemma 12, this implies that the function 
\begin{equation*}
q(z)=\frac{1}{2}+\sum_{m=2}^{\infty }\dfrac{4\left( \gamma -\lambda \right) 
}{m^{2}\left[ 2\gamma +\left( \delta -\gamma \right) (m-1)\right] }z^{m-1}
\end{equation*}%
is analytic and ${\text{Re}}\{q(z)\}>0$ in $\mathcal{U}.$ Writing%
\begin{equation*}
\frac{F(z)}{z}=p(z)\ast \left( 1+\sum_{m=2}^{\infty }\dfrac{4\left( \gamma
-\lambda \right) }{m^{2}\left[ 2\gamma +\left( \delta -\gamma \right) (m-1)%
\right] }z^{m-1}\right) ,
\end{equation*}%
and making use of Lemma 13 gives that ${\text{Re}}\left \{ \dfrac{F(z)}{z}%
\right \} >\dfrac{1}{2}$ for $z\in \mathcal{U}.$
\end{proof}

\begin{lemma}
Let $F_{i}\in \mathcal{R}(\gamma ,\delta ,\lambda )$ for $i=1,2.$ Then $%
F_{1}\ast F_{2}$ $\in $ $\mathcal{R}(\gamma ,\delta ,\lambda ).$
\end{lemma}

\begin{proof}
Suppose $F_{1}(z)=z+\sum_{m=2}^{\infty }A_{m}z^{m}$ \ and \ $%
F_{2}(z)=z+\sum_{m=2}^{\infty }B_{m}z^{m}.$ Then the convolution of $%
F_{1}(z) $ and $F_{2}(z)$ is defined by%
\begin{equation*}
F(z)=(F_{1}\ast F_{2})(z)=z+\sum_{m=2}^{\infty }A_{m}B_{m}z^{m}.
\end{equation*}%
Since $F^{\prime }(z)=F_{1}^{\prime }(z)\ast \frac{F_{2}(z)}{z}$, $%
zF^{\prime \prime }(z)=zF_{1}^{\prime \prime }(z)\ast \frac{F_{2}(z)}{z}$\
and $zF^{\prime \prime \prime }(z)=zF_{1}^{\prime \prime \prime }(z)\ast 
\frac{F_{2}(z)}{z}$ then we have%
\begin{eqnarray}
&&\frac{2\gamma F^{\prime }(z)+2\delta zF^{\prime \prime }(z)+\left( \delta
-\gamma \right) z^{2}F^{\prime \prime \prime }\left( z\right) -2\lambda }{%
2\left( \gamma -\lambda \right) }  \notag \\
&=&\left( \frac{2\gamma F_{1}^{\prime }(z)+2\delta zF_{1}^{\prime \prime
}(z)+\left( \delta -\gamma \right) z^{2}F_{1}^{\prime \prime \prime }\left(
z\right) -2\lambda }{2\left( \gamma -\lambda \right) }\right) \ast \dfrac{%
F_{2}(z)}{z}.  \label{eqconv}
\end{eqnarray}%
Since $F_{1}\in $ $\mathcal{R}(\gamma ,\delta ,\lambda ),$ 
\begin{equation*}
{\text{Re}}\left \{ \frac{2\gamma F_{1}^{\prime }(z)+2\delta zF_{1}^{\prime
\prime }(z)+\left( \delta -\gamma \right) z^{2}F_{1}^{\prime \prime \prime
}\left( z\right) -2\lambda }{2\left( \gamma -\lambda \right) }\right \} >0%
\text{ }\left( z\in \mathcal{U}\right)
\end{equation*}%
and using Lemma 14, ${\text{Re}}\left \{ \dfrac{F_{2}(z)}{z}\right \} >%
\dfrac{1}{2}$ \ in $\mathcal{U}.$ Now applying Lemma 13 to (\ref{eqconv})
yields ${\text{Re}}\left( \frac{2\gamma F^{\prime }(z)+2\delta zF^{\prime
\prime }(z)+\left( \delta -\gamma \right) z^{2}F^{\prime \prime \prime
}\left( z\right) -2\lambda }{2\left( \gamma -\lambda \right) }\right) >0$ in 
$\mathcal{U}.$ Thus, $F=F_{1}\ast F_{2}$ $\in $ $\mathcal{R}(\gamma ,\delta
,\lambda ).$
\end{proof}

Now using Lemma 15, we prove that the class $\mathcal{R}_{H}^{0}(\gamma
,\delta ,\lambda )$ is closed under convolutions of its members. We make use
of the techniques and methodology introduced by Dorff \cite{Dorff} for
convolution.

\begin{theorem}
Let $\mathfrak{f}_{i}\in \mathcal{R}_{H}^{0}(\gamma ,\delta ,\lambda )$ for $%
i=1,2.$ Then $\mathfrak{f}_{1}\ast \mathfrak{f}_{2}$ $\in $ $\mathcal{R}%
_{H}^{0}(\gamma ,\delta ,\lambda ).$
\end{theorem}

\begin{proof}
Suppose $\mathfrak{f}_{i}=\mathfrak{s}_{i}+\overline{\mathfrak{t}_{i}}\in 
\mathcal{R}_{H}^{0}(\gamma ,\delta ,\lambda )$ $(i=1,2)$. Then the
convolution of $\mathfrak{f}_{1}$ and $\mathfrak{f}_{2}$ is defined as $%
\mathfrak{f}_{1}\ast \mathfrak{f}_{2}=\mathfrak{s}_{1}\ast \mathfrak{s}_{2}+%
\overline{\mathfrak{t}_{1}\ast \mathfrak{t}_{2}}.$ In order to prove that $%
\mathfrak{f}_{1}\ast \mathfrak{f}_{2}$ $\in $ $\mathcal{R}_{H}^{0}(\gamma
,\delta ,\lambda ),$ we need to prove that $F_{\epsilon }=\mathfrak{s}%
_{1}\ast \mathfrak{s}_{2}+\epsilon (\mathfrak{t}_{1}\ast \mathfrak{t}%
_{2})\in \mathcal{R}(\gamma ,\delta ,\lambda )$ for each $\epsilon
(|\epsilon |=1).$ By Lemma 15, the class $\mathcal{R}(\gamma ,\delta
,\lambda )$ is closed under convolutions for each $\epsilon (|\epsilon |=1),$
$\mathfrak{s}_{i}+\epsilon \mathfrak{t}_{i}\in \mathcal{R}(\gamma ,\delta
,\lambda )$ for $i=1,2.$ Then both $F_{1}$ and $~F_{2}$ given by%
\begin{equation*}
F_{1}=(\mathfrak{s}_{1}-\mathfrak{t}_{1})\ast (\mathfrak{s}_{2}-\epsilon 
\mathfrak{t}_{2})\text{ \ and \ }F_{2}=(\mathfrak{s}_{1}+\mathfrak{t}%
_{1})\ast (\mathfrak{s}_{2}+\epsilon \mathfrak{t}_{2}),
\end{equation*}%
belong to $\mathcal{R}(\gamma ,\delta ,\lambda )$. Since $\mathcal{R}(\gamma
,\delta ,\lambda )$ is closed under convex combinations, then the function%
\begin{equation*}
F_{\epsilon }=\frac{1}{2}(F_{1}+F_{2})=\mathfrak{s}_{1}\ast \mathfrak{s}%
_{2}+\epsilon (\mathfrak{t}_{1}\ast \mathfrak{t}_{2})
\end{equation*}%
belongs to $\mathcal{R}(\gamma ,\delta ,\lambda )$. Hence $\mathcal{R}%
_{H}^{0}(\gamma ,\delta ,\lambda )$ is closed under convolution.
\end{proof}

Now we consider the Hadamard product of a harmonic function with an analytic
function which is defined by Goodloe \cite{Goodloe} as 
\begin{equation*}
\mathfrak{f}\widetilde{\ast }\varphi =\mathfrak{s}\ast \varphi +\overline{%
\mathfrak{t}\ast \varphi },
\end{equation*}%
where $\mathfrak{f}=\mathfrak{s}+\overline{\mathfrak{t}}$ is harmonic
function and $\varphi $ is an analytic function in $\mathcal{U}.$

\begin{theorem}
Let $\mathfrak{f}\in $ $\mathcal{R}_{H}^{0}(\gamma ,\delta ,\lambda )$ and $%
\varphi \in \mathcal{A}$ be such that ${\text{Re}}\left( \dfrac{\varphi (z)}{%
z}\right) >\dfrac{1}{2}$ for $z\in \mathcal{U},$ then $\mathfrak{f}%
\widetilde{\ast }\varphi $ $\in $ $\mathcal{R}_{H}^{0}(\gamma ,\delta
,\lambda ).$
\end{theorem}

\begin{proof}
Suppose that $\mathfrak{f}=\mathfrak{s}+\overline{\mathfrak{t}}\in $ $%
\mathcal{R}_{H}^{0}(\gamma ,\delta ,\lambda ),$ then $F_{\epsilon }=%
\mathfrak{s}+\epsilon \mathfrak{t}\in $ $\mathcal{R}(\gamma ,\delta ,\lambda
)$ for each $\epsilon (|\epsilon |=1).$ By Theorem 2, in order to show that $%
\mathfrak{f}\widetilde{\ast }\varphi $ $\in $ $\mathcal{R}_{H}^{0}(\gamma
,\delta ,\lambda ),$ we need to show that $G=\mathfrak{s}\ast \varphi
+\epsilon (\mathfrak{t}\ast \varphi )$ $\in $ $\mathcal{R}(\gamma ,\delta
,\lambda )$ for each $\epsilon (|\epsilon |=1).$ Write $G$ as $G=F_{\epsilon
}\ast \varphi ,$ and 
\begin{eqnarray*}
&&\frac{1}{2\left( \gamma -\lambda \right) }\left( 2\gamma G^{\prime
}(z)+2\delta zG^{\prime \prime }(z)+\left( \delta -\gamma \right)
z^{2}G^{\prime \prime \prime }\left( z\right) -2\lambda \right) \\
&=&\frac{1}{2\left( \gamma -\lambda \right) }\left( 2\gamma F_{\epsilon
}^{\prime }(z)+2\delta zF_{\epsilon }^{\prime \prime }(z)+\left( \delta
-\gamma \right) z^{2}F_{\epsilon }^{\prime \prime \prime }\left( z\right)
-2\lambda \right) \ast \dfrac{\varphi (z)}{z}.
\end{eqnarray*}%
Since ${\text{Re}}\left( \dfrac{\varphi (z)}{z}\right) >\dfrac{1}{2}$ \ and
\ ${\text{Re}}\{2\gamma F_{\epsilon }^{\prime }(z)+2\delta zF_{\epsilon
}^{\prime \prime }(z)+\left( \delta -\gamma \right) z^{2}F_{\epsilon
}^{\prime \prime \prime }\left( z\right) -2\lambda \}>0$ in $\mathcal{U},$
Lemma 13 proves that $G\in \mathcal{R}(\gamma ,\delta ,\lambda )$.
\end{proof}

\begin{corollary}
Let $\mathfrak{f}\in $ $\mathcal{R}_{H}^{0}(\gamma ,\delta ,\lambda )$ and $%
\varphi \in \mathcal{K},$ then $f\widetilde{\ast }\varphi $ $\in $ $\mathcal{%
R}_{H}^{0}(\gamma ,\delta ,\lambda ).$
\end{corollary}

\begin{proof}
Suppose $\varphi \in \mathcal{K},$ then ${\text{Re}}\left( \dfrac{\varphi (z)%
}{z}\right) >\dfrac{1}{2}$ for $z\in \mathcal{U}.$ As a corollary of Theorem
17, $f\widetilde{\ast }\varphi $ $\in $ $\mathcal{R}_{H}^{0}(\gamma ,\delta
,\lambda ).$
\end{proof}

\section{Radii of fully convexity, starlikeness, uniformly convexity and
uniformly starlikeness}

In this section, we obtain the radii of fully convexity and starlikeness of
the class $\mathcal{R}_{H}^{0}(\gamma ,\delta ,\lambda ).$ Estimates on $%
\lambda $ are also given that would ensure functions in $\mathcal{R}%
_{H}^{0}(1,\delta ,\lambda )$ is fully convex.

The proof of the theorems follow from the proof of the results from \cite%
{Ghosh-fully, Nagpal-fully}. First, we state the following lemmas give
sufficient conditions for functions $\mathfrak{f}$ in $\mathcal{H}^{0}$ to
belong to $\mathcal{FK}_{H}^{0}$ and $\mathcal{FS}_{H}^{\ast ,0}$
respectively.

\begin{lemma}
Let $\mathfrak{f=s+}\overline{\mathfrak{t}}$, where $\mathfrak{s}$ and $%
\mathfrak{t}$ are given by (\ref{eqH}). Further, let%
\begin{equation}
\sum \limits_{m=2}^{\infty }m^{2}\left[ \left \vert a_{m}\right \vert +\left
\vert b_{m}\right \vert \right] \leq 1.  \label{FKH}
\end{equation}%
Then $\mathfrak{f}$ is harmonic univalent in $\mathcal{U}$, and $\mathfrak{f}%
\in $ $\mathcal{FK}_{H}^{0}$.
\end{lemma}

\begin{lemma}
Let $\mathfrak{f=s+}\overline{\mathfrak{t}}$, where $\mathfrak{s}$ and $%
\mathfrak{t}$ are given by (\ref{eqH}). Further, let%
\begin{equation}
\sum \limits_{m=2}^{\infty }m\left[ \left \vert a_{m}\right \vert +\left
\vert b_{m}\right \vert \right] \leq 1.  \label{FSH}
\end{equation}%
Then $\mathfrak{f}$ is harmonic univalent in $\mathcal{U}$, and $\mathfrak{f}%
\in $ $\mathcal{FS}_{H}^{\ast ,0}$.
\end{lemma}

The following identities are useful in the proof of the theorems:

\begin{eqnarray}
(i)\text{ }\sum \limits_{m=2}^{\infty }mr^{m-1} &=&\frac{r\left( 2-r\right) 
}{\left( 1-r\right) ^{2}},  \label{id1} \\
(ii)\text{ }\sum \limits_{m=2}^{\infty }m^{2}r^{m-1} &=&\frac{r\left(
4-3r+r^{2}\right) }{\left( 1-r\right) ^{3}}.  \label{id3}
\end{eqnarray}

\begin{theorem}
Let $\mathfrak{f}=\mathfrak{s}+\overline{\mathfrak{t}}\in \mathcal{R}%
_{H}^{0}(\gamma ,\delta ,\lambda ).$ Then $\mathfrak{f}$ is fully convex in $%
\left \vert z\right \vert <r_{c},$ where $r_{c}$ is the unique real root of $%
pc(r)=0$ in $(0,1),$ and where 
\begin{equation*}
pc(r)=\left( -\delta -2\gamma +\lambda \right) r^{3}+\left( 3\delta +6\gamma
-3\lambda \right) r^{2}+\left( -3\delta -7\gamma +4\lambda \right) r+\delta
+\gamma .
\end{equation*}
\end{theorem}

\begin{proof}
Let $\mathfrak{f}=\mathfrak{s}+\overline{\mathfrak{t}}\in \mathcal{R}%
_{H}^{0}(\gamma ,\delta ,\lambda )$ where $\mathfrak{s}\left( z\right)
=z+\sum \limits_{m=2}^{\infty }a_{m}z^{m}$ and $\mathfrak{t}\left( z\right)
=\sum \limits_{m=2}^{\infty }b_{m}z^{m}.$ For $r\in (0,1)$, it is sufficient
to show that $\mathfrak{f}_{r}\in \mathcal{FK}_{H}^{0}$ where%
\begin{equation*}
\mathfrak{f}_{r}\left( z\right) =\frac{\mathfrak{f}\left( rz\right) }{r}%
=z+\sum \limits_{m=2}^{\infty }a_{m}r^{m-1}z^{m}+\overline{\sum
\limits_{m=2}^{\infty }b_{m}r^{m-1}z^{m}}.
\end{equation*}%
Consider this time, the sum%
\begin{equation}
S=\sum \limits_{m=2}^{\infty }m^{2}\left( \left \vert a_{m}\right \vert
+\left \vert b_{m}\right \vert \right) r^{m-1}.  \label{a}
\end{equation}%
In view of Theorem 7 (i) and (\ref{id3}), (\ref{a}) gives 
\begin{eqnarray*}
S &\leq &\sum \limits_{m=2}^{\infty }m^{2}\left( \frac{4\left( \gamma
-\lambda \right) }{m^{2}\left[ 2\gamma +\left( \delta -\gamma \right) (m-1)%
\right] }\right) r^{m-1} \\
&\leq &\frac{\gamma -\lambda }{\delta +\gamma }\sum \limits_{m=2}^{\infty
}m^{2}r^{m-1} \\
&=&\frac{\gamma -\lambda }{\delta +\gamma }\frac{r\left( 4-3r+r^{2}\right) }{%
\left( 1-r\right) ^{3}}=:X_{1.}
\end{eqnarray*}%
Lemma 19 implies that in order to show that $\mathfrak{f}_{r}\in \mathcal{FK}%
_{H}^{0}$, it is sufficient to show that $X_{1}\leq 1.$ Thus, we need to
prove that $\left( -\delta -2\gamma +\lambda \right) r^{3}+\left( 3\delta
+6\gamma -3\lambda \right) r^{2}+\left( -3\delta -7\gamma +4\lambda \right)
r+\delta +\gamma \geq 0.$

Suppose $pc(r):=\left( -\delta -2\gamma +\lambda \right) r^{3}+\left(
3\delta +6\gamma -3\lambda \right) r^{2}+\left( -3\delta -7\gamma +4\lambda
\right) r+\delta +\gamma ,$ so that $X_{1}\leq 1$ whenever $pc(r)\geq 0.$ It
is easy to observe that $pc(0)=\delta +\gamma >0$ and $pc(1)=2\left( \lambda
-\gamma \right) <0$, and hence $pc(r)$ has at least one root in $\left(
0,1\right) .$

To show that $pc(r)$ has exactly one root in $\left( 0,1\right) ,$ it is
sufficient to prove that $pc(r)$ is monotonic function on $\left( 0,1\right)
.$ A simple computation shows that 
\begin{eqnarray*}
pc^{\prime }(r) &=&\left( -6\gamma +3\lambda -3\delta \right) r^{2}+\left(
12\gamma -6\lambda +6\delta \right) r-3\delta -7\gamma +4\lambda \\
pc^{\prime }(0) &=&-3\delta -7\gamma +4\lambda =-3\left( \gamma +\delta
\right) -4\left( \gamma -\lambda \right) <0 \\
pc^{\prime }(1) &=&\lambda -\gamma <0 \\
pc^{\prime \prime }(r) &=&\left( -6\delta -12\gamma +6\lambda \right)
r+6\delta +12\gamma -6\lambda \\
&=&\left[ -6\left( \gamma +\delta \right) -6\left( \delta -\lambda \right) %
\right] r-\left[ -6\left( \gamma +\delta \right) -6\left( \delta -\lambda
\right) \right] \\
&=&\left[ -6\left( \gamma +\delta \right) -6\left( \delta -\lambda \right) %
\right] \left( r-1\right) >0\text{ \ for }r\in \left( 0,1\right) .
\end{eqnarray*}%
Hence $pc^{\prime }(r)$ is a strictly monotonic increasing function on $%
\left( 0,1\right) $. Since $pc^{\prime }(1)<0$, we conclude that $pc^{\prime
}(r)<0$ on $\left( 0,1\right) $. This shows that $pc(r)$ is strictly
monotonically decreasing on $\left( 0,1\right) $. Thus $pc(r)=0$ has exactly
one root in $\left( 0,1\right) $. Since $pc(r)$ is strictly monotonically
decreasing on $\left( 0,1\right) $ with $pc(0)>0$ and $pc(r_{c})=0,$ it is
easy to see that $pc(r)\geq 0$ for $0<r\leq r_{c}$. Hence $\mathfrak{f}$ is
fully convex in $\left \vert z\right \vert <r_{c}$.
\end{proof}

\begin{theorem}
Let $\mathfrak{f}=\mathfrak{s}+\overline{\mathfrak{t}}\in \mathcal{R}%
_{H}^{0}(\gamma ,\delta ,\lambda )$. Then $\mathfrak{f}$ is fully starlike
in $\left \vert z\right \vert <r_{s},$ where $r_{s}$ is the unique real root
of $ps(r)=0$ in $(0,1),$ and where 
\begin{equation*}
ps(r)=\left( \delta +2\gamma -\lambda \right) r^{2}+\left( -2\delta -4\gamma
+2\lambda \right) r+\delta +\gamma .
\end{equation*}
\end{theorem}

\begin{proof}
Let $\mathfrak{f}=\mathfrak{s}+\overline{\mathfrak{t}}\in \mathcal{R}%
_{H}^{0}(\gamma ,\delta ,\lambda ).$ For $r\in (0,1)$, let%
\begin{equation*}
\mathfrak{f}_{r}\left( z\right) =\frac{\mathfrak{f}\left( rz\right) }{r}%
=z+\sum \limits_{m=2}^{\infty }a_{m}r^{m-1}z^{m}+\overline{\sum
\limits_{m=2}^{\infty }b_{m}r^{m-1}z^{m}}
\end{equation*}%
Consider the sum%
\begin{equation}
S=\sum \limits_{m=2}^{\infty }m\left( \left \vert a_{m}\right \vert +\left
\vert b_{m}\right \vert \right) r^{m-1}  \label{b}
\end{equation}%
Using Theorem 7(i) and (\ref{id1}), (\ref{b}) gives 
\begin{eqnarray*}
S &\leq &\sum \limits_{m=2}^{\infty }m\left( \frac{4\left( \gamma -\lambda
\right) }{m^{2}\left[ 2\gamma +\left( \delta -\gamma \right) (m-1)\right] }%
\right) r^{m-1} \\
&\leq &\frac{\gamma -\lambda }{\delta +\gamma }\sum \limits_{m=2}^{\infty
}mr^{m-1} \\
&=&\frac{\gamma -\lambda }{\delta +\gamma }\frac{r\left( 2-r\right) }{\left(
1-r\right) ^{2}}=:X_{2.}
\end{eqnarray*}%
In view of Lemma 19 in order to prove that $\mathfrak{f}_{r}\in \mathcal{FSH}%
^{\ast ,0}$, it is sufficient to show that $X_{2}\leq 1$. This implies that
it suffices to show $\left( \delta +2\gamma -\lambda \right) r^{2}+\left(
-2\delta -4\gamma +2\lambda \right) r+\delta +\gamma \geq 0.$

Suppose $ps(r)=\left( \delta +2\gamma -\lambda \right) r^{2}+\left( -2\delta
-4\gamma +2\lambda \right) r+\delta +\gamma ,$ so that $X_{2}\leq 1$
whenever $ps(r)\geq 0.$ It is easy to observe that $ps(0)=\delta +\gamma >0$
and $ps(1)=\lambda -\gamma <0$, and hence $ps(r)$ has a real root in $\left(
0,1\right) .$

To show that $ps(r)$ has exactly one root in $\left( 0,1\right) ,$ it is
sufficient to prove that $ps(r)$ is monotonic function on $\left( 0,1\right)
.$ A simple computation shows that 
\begin{eqnarray*}
ps^{\prime }(r) &=&2\left( \delta +2\gamma -\lambda \right) \left(
r-1\right) <0,\text{\ \ for }r\in \left( 0,1\right) \\
ps^{\prime }(0) &=&2\left( \lambda -\delta -2\gamma \right) , \\
ps^{\prime }(1) &=&0, \\
ps^{\prime \prime }(r) &=&2\left( \delta +2\gamma -\lambda \right) >0.
\end{eqnarray*}%
Hence $ps^{\prime }(r)$ is a strictly monotonically increasing function on $%
\left( 0,1\right) $. Since $ps^{\prime }(1)=0$, we conclude that $ps^{\prime
}(r)<0$ on $\left( 0,1\right) $. This shows that $ps(r)$ is strictly
monotonically decreasing on $\left( 0,1\right) $. Thus $ps(r)$ has exactly
one root in $\left( 0,1\right) $. Since $ps(r)$ is strictly monotonically
decreasing on $\left( 0,1\right) $ with $ps(0)>0$ and $ps(r_{s})=0$, it is
easy to see that $ps(r)\geq 0$ for $0<r\leq r_{s}$. Hence $\mathfrak{f}$ is
fully starlike in $\left \vert z\right \vert <r_{s}$.
\end{proof}

\begin{lemma}
\cite[Corollary 3.2]{Rosihan2018} Let $\lambda <1,$ $\delta \geq 1,$ and $%
\mathfrak{f}\in \mathcal{R}(1,\delta ,\lambda )$. If $\lambda $ satisfies 
\begin{equation}
7-3\delta =4\lambda +4(1-\lambda )\sum \limits_{m=1}^{\infty }\frac{%
2m(3-\delta )+\left( \delta -5\right) }{\left( m+1\right) \left( m\left(
\delta -1\right) +2\right) },  \label{eqconvexRosihan}
\end{equation}%
then $\mathfrak{f}$ is convex in $\mathcal{U}.$
\end{lemma}

\begin{theorem}
Suppose $\mathfrak{f}\in \mathcal{R}_{H}^{0}(1,\delta ,\lambda )$ with $%
\lambda <1,$ $\delta \geq 1.$ If $\lambda $ satisfies $\left( \ref%
{eqconvexRosihan}\right) $, then $\mathfrak{f}$ is fully convex in $\mathcal{%
U}.$
\end{theorem}

\begin{proof}
Let $\lambda <1,$ $\delta \geq 1$ and $\mathfrak{f}\in \mathcal{R}%
_{H}^{0}(1,\delta ,\lambda ).$ Then $F_{\epsilon }=\mathfrak{s}+\epsilon 
\mathfrak{t}\in \mathcal{R}(1,\delta ,\lambda )$ for each $\epsilon \left(
\left \vert \epsilon \right \vert =1\right) .$ If $\lambda $ satisfies (\ref%
{eqconvexRosihan}), then $F_{\epsilon }$ is convex in $\mathcal{U}.$ In view
of \cite[Corollary 2.4]{Nagpaljkms}, it follows that $\mathfrak{f}$ is fully
convex in $\mathcal{U}.$
\end{proof}

\begin{problem}
Find $\gamma $, $\delta $ and $\lambda $ so that functions $\mathfrak{f}\in 
\mathcal{R}_{H}^{0}(\gamma ,\delta ,\lambda )$ are fully convex in $\mathcal{%
U}$.
\end{problem}

\end{document}